\documentclass[twoside,a4,11pt]{amsart}

\usepackage[english]{babel}
\usepackage[utf8x]{inputenc}
\usepackage{amsmath}
\usepackage{graphicx}
\usepackage[colorinlistoftodos]{todonotes}
\usepackage{amsmath,amscd,amsthm}
\usepackage{amstext, amsfonts, a4}
\usepackage{amssymb}
\usepackage{tikz-cd}
\usepackage{fancyhdr}
\usepackage{enumerate}
\usepackage{cleveref}
\usepackage{xcolor}

\numberwithin{equation}{section}
\newtheorem{thm}[equation]{Theorem}

\newtheorem{prop}[equation]{Proposition}
\newtheorem{cor}[equation]{Corollary}
\newtheorem{lem}[equation]{Lemma}

\theoremstyle{definition}

\newtheorem{ex}[equation]{Example}

\newtheorem{rem}[equation]{Remark}

\renewcommand{\dim}{\operatorname{\mathsf{dim}}}
\renewcommand{\deg}{\operatorname{\mathsf{deg}}}

\newcommand\ind{\operatorname{\mathsf{ind}}}
\renewcommand\exp{\operatorname{\mathsf{exp}}}
\newcommand\op{\operatorname{\mathsf{op}}}
\newcommand\End{\operatorname{\mathsf{End}}}
\newcommand\Br{\operatorname{\mathsf{Br}}}

\newcommand\id{\operatorname{\mathsf{id}}}

\newcommand\Int{\operatorname{\mathsf{Int}}}

\newcommand\corr{\operatorname{\mathsf{cor}}}

\newcommand{\cchar}{\mathsf{char}}
\newcommand{\can}{\operatorname{\mathsf{can}}}

\newcommand{\mg}[1]{{#1}^{\times}}

\newcommand{\s}{\sigma}
\newcommand{\nat}{\mathbb{N}}

\renewcommand{\leq}{\leqslant}
\renewcommand{\geq}{\geqslant}

\renewcommand{\setminus}{\smallsetminus}

\newcommand\Ad{\operatorname{\mathsf{Ad}}}
\newcommand\ad{\operatorname{\mathsf{ad}}}
\newcommand\sw{\operatorname{\mathsf{sw}}}
\newcommand\rk{\operatorname{\mathsf{rk}}}
\newcommand{\C}{\mathsf{C}}
\newcommand{\Z}{\mathsf{Z}}

\renewcommand{\setminus}{\smallsetminus}
\renewcommand{\sup}{\mathsf{sup}}

\newcommand{\I}{\mathsf{I}}
\newcommand{\Iq}[2]{\I_{\mathsf{q}}^{#1}{#2}}

\newcommand{\q}{\mathsf{q}}

\newcommand{\sfrac}[2]{\mbox{$\frac{#1}{#2}$}}
\newcommand{\qq}{\mathbb{Q}}

\begin{document}
\title{Hermitian $u$-invariants under quadratic field extensions}
	
\date{January 13, 2025}
	
\author{Karim Johannes Becher}
\author{Fatma Kader B\.{i}ng\"{o}l}
		
\address{Universiteit Antwerpen, Departement Wiskunde, Middelheim\-laan~1, 2020 Ant\-werpen, Belgium.}
\email{karimjohannes.becher@uantwerpen.be}

\address{Scuola Normale Superiore, Piazza dei Cavalieri 7, 56126 Pisa, Italia}
\email{fatmakader.bingol@sns.it}
	
\begin{abstract}
    The hermitian $u$-invariants of a central simple algebra with involution are studied. In this context, a new technique is obtained to give bounds for the behavior of these invariants under a quadratic field extension.
    This is applied to obtain bounds in terms of the index and the $u$-invariant of the base field.

 
\medskip\noindent
{\sc Classification} (MSC 2020): 11E39, 12E15, 16K20
\end{abstract}
\maketitle
	
\section{Introduction}
The concept of the $u$-invariant naturally extends from quadratic forms over fields to hermitian forms over central simple algebras with an involution.
This paper is devoted to the problem of bounding the hermitian $u$-invariants of central simple algebras of exponent $2$.
Recall that such algebras can be represented as a tensor product of a finite number of quaternion algebras.
In the presence of a suitable separable quadratic subfield, one can bound the hermitian $u$-invariant of a central simple algebra with involution in terms of the hermitian $u$-invariant of a subalgebra stable under the involution.
This relies on a method used in characteristic different from $2$ by E. Bayer-Fluckiger's and R. Parimala in the construction of an exact sequence of Witt groups of hermitian forms \cite[Appendix 2 and \S 3.1]{BFP95}, which plays an essential role in their classification results for hermitian forms over central simple algebras with involution over fields of cohomological dimension $2$.
It is used later in \cite{Mah05}, \cite{PS13} and \cite{Wu18} to obtain upper bounds for the hermitian $u$-invariants of a central simple algebra of exponent $2$ in terms of the $u$-invariant of the base field. 
In this article, we refine this method and apply it to study the behaviour of the hermitian $u$-invariants under $2$-extensions and in particular multiquadratic extensions. 
In this way, we obtain bounds on the hermitian $u$-invariants of a central simple algebra with involution in terms of the degree and the $u$-invariant of a splitting $2$-extension (\Cref{2-ext-unitary-hermit} and \Cref{orthogonal-hermit-split-multi}). 
For certain algebras with involution of small index, we obtain in \Cref{C:unitary-biquat} and \Cref{ort-hermit-ind8-wu+our} improvements to previously existing bounds.

We denote by $\nat$ the set of natural numbers including $0$ and set $\nat^+=\nat\setminus\{0\}$.

\section{Hermitian forms and involutions} 

Our main references are \cite{Knu91} for the theory of hermitian forms and \cite{KMRT98} for the theory of algebras with involution.
The notions and basic facts from the theory of central simple algebras that we need are mostly covered by \cite[Chap.~8]{Scha85}.

Throughout this article let $F$ denote a field. Let $A$ be a central simple $F$-algebra.
The \emph{degree},  \emph{index} and \emph{exponent of $A$} are denoted by $\deg A$, $\ind A$ and $\exp A$, respectively. 
Given another central simple $F$-algebra $B$, we write $A\sim B$ to indicate that $A$ and $B$  are Brauer equivalent.
By Wedderburn's Theorem \cite[Chap.~8, Cor. 1.6]{Scha85}, we have  $A\simeq\mathbb{M}_s(D)$ for a unique central $F$-division algebra $D$ and a unique $s\in\nat^+$.
Then $\ind{A}=\deg{D}$ and $\deg{A}=s\cdot \ind{A}$. We say that $A$ is \emph{split} if $D=F$, or equivalently, if $\ind A=1$.
By \cite[Chap.~8, Theorem 1.8]{Scha85}, every finitely generated $A$-right module $V$ decomposes into a direct sum of simple $A$-right modules. 
The number of simple components in a decomposition of $V$ as a direct sum of simple $A$-right modules is called the \emph{rank of} $V$ (over $A$) and denoted by $\rk_A V$.
It follows that $\rk_A V=\frac{\dim_FV}{s\cdot\dim_FD}$.
In particular, we have $\rk_A A=s$, and if $A$ is a division algebra, then $\rk_A V$ is the dimension of $V$ as an $A$-right vector space.
Given a field extension $K/F$, we obtain from $A$ a central simple $K$-algebra $A_K=A\otimes_FK$.

An \emph{involution on} a ring $A$ is an anti-auto\-morphism $\sigma: A\to A$ such that $\sigma^2=\id_A$.
If $A$ is an $F$-algebra, $\s$ is an involution on $A$,  and $K/F$ is a field extension, then $\s\otimes \id_K$ is an involution on $A_K=A\otimes_FK$, which we denote by $\sigma_K$.
If $K/F$ is a quadratic \'etale extension (that is, either $K\simeq F\times F$ or $K/F$ is a separable quadratic field extension), then the non-trivial 
$F$-automorphism of $K$ is an involution on $K$, which we call  the \emph{canonical involution of $K/F$} and which we denote by $\can_{K/F}$.

Let now $A$ be a central simple $F$-algebra and $\sigma$ an involution on $A$ which is $F$-linear.
We fix a field extension $K/F$ such that $A_K$ is split.
Hence $A_K\simeq\End_KV$ for a $K$-vector space $V$ with $\dim_K V=\deg A$, and under this isomorphism of $K$-algebras, $\s_K$ corresponds to a $K$-linear involution on $\End_KV$, which is adjoint to a (nonsingular) alternating or symmetric $K$-bilinear form $b$ on $V$. 
Moreover, whether $b$ is alternating depends neither on the choice of the field extension $K/F$ nor on the isomorphism; see \cite[Prop.~2.6]{KMRT98}.
We call the involution $\s$ \emph{symplectic} if the bilinear form $b$ is alternating, and we call it \emph{orthogonal} otherwise.

For a ring $A$, we set $\Z(A)=\{x\in A\mid xy=yx\mbox{ for all }y\in A\}$, which is a subring called the \emph{center of $A$}.
By an \emph{$F$-algebra with involution}, we mean a pair $(A,\s)$ where $A$ is a finite-dimensional $F$-algebra and $\s$ is an $F$-linear involution on $A$ such that $F=\{x\in \Z(A)\mid \s(x)=x\}$ and $A$ has no non-trivial two-sided ideal $I$ with $\s(I)= I$. 
There are two kinds of situations for this to occur:
\begin{enumerate}
    \item $A$ is a central simple $F$-algebra and $\s$ is an $F$-linear involution. 
    \item $K=\Z(A)$ is a quadratic \'etale extension of $A$ and $\s|_K$ is an automorphism of order $2$ of $K$ with $F$ as its fixed field. 
    In this case, either $K/F$ is a separable quadratic field extension and $A$ is a central simple $K$-algebra, or $K\simeq F\times F$ and $A\simeq B\times B^{\op}$ for a central simple $F$-algebra $B$ and where $B^{\op}$ denotes its opposite algebra (which coincides with $B$ as a set), and $\s$ corresponds to the so-called \emph{switch map} $\sw_B$ on $B\times B^{\op}$, which is defined by $\sw_B(b,b')=(b',b)$ for $b,b'\in B$.
\end{enumerate}

If $K$ is either a field or a product of two copies of a field, then an involution $\s$ on a $K$-algebra $A$ is called \emph{unitary} if $\s|_K\neq \id_K$, and this case $(A,\s)$ is an $F$-algebra with involution for $F=\{x\in K\mid \s(x)=x\}$. 
We then also call $\s$ a \emph{$K/F$-unitary involution on $A$}.
Note that, while the fixed field $F$ in this situation is determined by $\s$, a central simple $K$-algebra $A$ can have many unitary involutions $\s$ with different fixed fields for $\s$, and this situation will also occur crucially in our study.

A unitary involution $\s$ on an $F$-algebra $A$ with $\Z(A)\simeq F\times F$ is also called \emph{unitary of inner type}. 
(The term is motivated by a relation to a corresponding notion for linear algebraic groups.)
For any $F$-algebra with involution $(A,\s)$ where $\s$ is not unitary of inner type, we have that $K=\Z(A)$ is a field, $K/F$ is separable with $[K:F]\leq 2$ and $A$ is a central simple $K$-algebra.

We recall the following criteria for the existence of involutions on a central simple algebra. 

\begin{thm}[Albert]\label{T:1st-ex-Albert}
    Let $A$ be a central simple $F$-algebra. There exists an orthogonal involution on $A$ if and only if $\exp A\leq 2$.
    There exists a symplectic involution on $A$ if and only if $\exp A\leq 2$ and $\deg A$ is even. 
\end{thm}
\begin{proof}
    By \cite[Theorem 3.1]{KMRT98}, an $F$-linear involution on $A$ exists if and only if $\exp A\leq 2$, and using \cite[Cor. 2.8]{KMRT98}, the statement follows.
\end{proof}

For a separable quadratic field extension $K/F$ and a central simple $K$-algebra $B$,
we denote by $\corr_{K/F}B$ the corestriction algebra of $B$ with respect to $K/F$ as defined in \cite[\S 8]{Dra83}.

\begin{thm}[Albert-Riehm-Scharlau]\label{T:ARS}
\label{exist-inv}
    Let $K/F$ be a separable quadratic field extension and let B be a central simple $K$-algebra. 
    There exists a $K/F$-unitary involution on $B$ if and only if $\corr_{K/F}B$ is split.
\end{thm}
\begin{proof}
    See \cite[Theorem 3.1]{KMRT98}.
\end{proof}  

It follows in particular from \Cref{T:1st-ex-Albert} and \Cref{T:ARS} that, if a central simple algebra $A$ has an involution $\s$, then every central simple algebra over $\Z(A)$ which is Brauer equivalent to $A$ carries an involution whose restriction to $\Z(A)$ is the same as for $\s$.

In the sequel, let $K$ be a field, let $B$ be a central simple $K$-algebra, and let $\gamma$ be an involution on $B$.
Then $F=\{x\in K\mid \gamma(x)=x\}$ is a subfield of $K$, and we obtain that $(B,\gamma)$ is an $F$-algebra with involution.

Let $V$ be a finitely generated $B$-right module and $\varepsilon\in\{\pm1\}$.
A bi-additive map $h:V\times V\to B$  is called an $\varepsilon$-\emph{hermitian form over} $(B,\gamma)$ if it satisfies the following:
\begin{itemize}
    \item $h(v\alpha,w\beta)=\gamma(\alpha)h(v,w)\beta$ for all $v,w\in V,\alpha,\beta\in B$,
    \item $h(w,v)=\varepsilon\gamma(h(v,w))$ for all $v,w\in V$.
\end{itemize} 
We also call $h$ simply a \emph{hermitian} (resp.~\emph{skew-hermitian}) \emph{form} when $\varepsilon=1$ (resp.~$\varepsilon=-1$).
For an $\varepsilon$-hermitian form $h$ over $(B,\gamma)$, we will denote by $\rk h$ the rank of the underlying $B$-right module $V$. 
We refer to \cite[Chap. I, \S 2.2 \& \S3.4]{Knu91} for the basic concepts of isometry ($\simeq$) and orthogonal sum ($\perp$) for $\varepsilon$-hermitian forms.

Consider an $\varepsilon$-hermitian form $h$ over $(B,\gamma)$ defined on the $B$-right module $V$.
For any $B$-right submodule $U$ of $V$, the restriction of $h$ to $U\times U$ defines an $\varepsilon$-hermitian form over $(B,\gamma)$ which we denote by $h|_U$.
A $B$-right submodule $U$ of $V$ is called \emph{totally isotropic} (\emph{with respect to $h$}) if $h|_U$ is the zero map.
The form $h$ is \emph{nonsingular} if for any $v\in V\setminus\{0\}$ there exists $w\in V$ such that $h(v,w)=0$.
The form $h$ is \emph{isotropic} if there exists some $v\in V\setminus\{0\}$ such that $h(v,v)=0$, and \emph{anisotropic} otherwise. 
The form $h$ is \emph{hyperbolic} if it is nonsingular 
and $V=U\oplus U'$ for two $B$-right submodules $U$ and $U'$ which are totally isotropic with respect to $h$.
In particular, any non-trivial hyperbolic $\varepsilon$-hermitian form is isotropic. 
If $\cchar K\neq 2$ or $\gamma|_K\neq \id_K$, then any rank-$2$ nonsingular isotropic $\varepsilon$-hermitian form over $(B,\gamma)$ is hyperbolic.

For a field extension $M/K$, we obtain a finitely generated $B_M$-right module $V_M=V\otimes_K M$ and an $\varepsilon$-hermitian form $h_M: V_M\times V_M\to B_M$ over $(B_M,\gamma_M)$ given by $h_M(v\otimes\alpha,w\otimes\beta)=h(v,w)\otimes\alpha\beta$ for  $v,w\in V,\alpha,\beta\in M$.

Let $V$ be a finitely generated $B$-right module. Let $h:V\times V\to B$ be a nonsingular hermitian or skew-hermitian form over $(B,\gamma)$. 
According to \cite[\S 4.A]{KMRT98}, $h$ determines an involution $\ad_h$ on $\End_BV$ satisfying
$$h(v,f(w))=h(\ad_h(f)(v),w)\quad \text{for all}\, v,w\in V\, \text{and all}\, f\in\End_BV.$$
We call $\ad_h$ the \emph{adjoint involution of} $h$.
Viewing $K$ naturally embedded into $\End_BV$, we have that $\ad_h|_K=\gamma|_K$. 
All involutions on $\End_BV$ arise in this way from some nonsingular hermitian form.

\begin{thm}\label{correspondence-(skew)hermit-involution}
\quad
\begin{enumerate}[$(a)$]
    \item Assume that $\cchar K\neq2$ and $\gamma|_K=\id_K$.
    Then any $K$-linear involution $\sigma$ on $\End_BV$ is the adjoint involution $\ad_h$ of some nonsingular hermitian or skew-hermitian form $h$ over $(B,\gamma)$, which is unique up to a factor in $F^{\times}$.
    Moreover, the involutions $\sigma$ and $\gamma$ are both orthogonal or both symplectic if $h$ is hermitian, whereas precisely one of them is orthogonal and the other one symplectic if $h$ is skew-hermitian.
    \item Assume that $\gamma|_K\neq\id_K$. Then any involution $\tau$ on $\End_BV$ with $\tau|_K=\gamma|_K$ is the adjoint involution $\ad_h$ for some nonsingular hermitian form $h$ over $(B,\gamma)$, which is unique up to a factor in $F^{\times}$.
\end{enumerate}
\end{thm}
\begin{proof}
    See \cite[Theorem 4.2]{KMRT98}.
\end{proof}

We denote by $\Ad_B(h)$ the $F$-algebra with involution $(\End_BV,\ad_h)$.

Let $(A,\sigma)$ be an $F$-algebra with involution.
We say that 
$\sigma$ (or $(A,\s)$) is \emph{isotropic} if there exists an $a\in A\setminus\{0\}$ such that $\sigma(a)a=0$, and \emph{anisotropic} otherwise. 
If there exists an element $e\in A$ such that $e^2=e$ and $\sigma(e)=1-e$, then $\s$ (or $(A,\s)$) is called \emph{hyperbolic}.
In particular, any hyperbolic involution is isotropic.

\begin{prop}\label{isotropic-hyperbolic-inv-equiv-cond}
    Let $K$ be a field, $D$ a central $K$-division algebra and $\gamma$ an involution on $D$.
    Assume that $\cchar K\neq2$ or $\gamma$ is unitary.
    Let $V$ be a finite-dimensional $D$-right vector space and $h: V\times V\to D$ an $\varepsilon$-hermitian form over $(D,\gamma)$.
    Then $\Ad_D(h)$ is isotropic (resp. hyperbolic) if and only if $h$ is isotropic (resp. hyperbolic).	
\end{prop}
\begin{proof}
    The statement for hyperbolicity is given in \cite[Prop. 6.7]{KMRT98}. We prove the statement for isotropy. 
	
    We set $(A,\sigma)=(\End_DV,\ad_h)$. 
    Assume that $(A,\sigma)$ is isotropic. Then there exists some $f\in A\setminus\{0\}$ such that $\sigma(f)\circ f=0$. 
    Since $f\neq0$, there exists $v\in V$ with $f(v)\neq0$. We have $h(f(v),f(v))=h((\sigma(f)\circ f)(v),v)=h(0,v)=0$. Thus $h$ is isotropic. 
	
    Conversely, assume that $h$ is isotropic. There exists some $v\in V\setminus\{0\}$ such that $h(v,v)=0$. 
    Since $h$ is nonsingular, we can find a nonzero vector $w\in V$ such that $h(v,w)\neq0$. 
    Consider $f:V\to V,u\mapsto vh(u,w)$. Then $f\in\End_DV$, and $f\neq0$ as $f(v)=vh(v,w)\neq0$. We have that 
\begin{equation*}
\begin{split}
    h((\sigma(f)\circ f)(u),u')&=h(f(u),f(u'))=h(vh(u,w),vh(u',w))\\
     &=\gamma(h(u,w))h(v,v)h(u',w)=0
\end{split}
\end{equation*}	
     for any $u,u'\in V$. Hence $\sigma(f)\circ f=0$. Therefore $(A,\sigma)$ is isotropic. 
\end{proof}

\section{Hermitian $u$-invariants} \label{section: hermit-u-invariant}
Let $K$ be a field and $B$ a central simple $K$-algebra.
Let $\varepsilon\in\{\pm 1\}$ and let $\gamma$ be an involution on $B$.
Following \cite[Chap.~9, Definition 2.4]{Pfi95}, we set
$$u(B,\gamma,\varepsilon)=\sup\{\rk h\mid \text{$h$ anisotropic $\varepsilon$-hermitian form over}\,(B,\gamma)\}\in\nat\cup\{\infty\}\,,$$
and we call $u(B,\gamma,\varepsilon)$ the $\varepsilon$-\emph{hermitian $u$-invariant} of $(B,\gamma)$.

For later use, we recall a well-known statement relating the $u$-invariant and the rank of an $\varepsilon$-hermitian form to the rank of totally isotropic subspaces.

\begin{lem}\label{L:ttis-subspace-u}
    Let $D$ be a central $K$-division algebra and $\gamma$ an involution on $D$.
    Let $\varepsilon\in\{\pm 1\}$ and let $h$ be an $\varepsilon$-hermitian form over $(D,\gamma)$.
    Then the $D$-right vector space on which $h$ is defined contains a totally isotropic subspace $U$ with respect to $h$ with $\rk U\geq \frac{1}2(\rk h-u(D,\gamma,\varepsilon))$.
\end{lem}
\begin{proof}
    Let $V$ be the $D$-right vector space on which $h$ is defined.
    Let $U$ be a maximal subspace of $V$ that is totally isotropic with respect to $h$. 
    We set $U'=\{w\in V\mid h(w,u)=0\mbox{ for all }u\in U\}$ and note that this is a $D$-right subspace of $V$ with $U\subseteq U'$ and $\rk U'\geq \rk V-\rk U$. 
    Then $U'= U\oplus U''$ for a $D$-right subspace $U''$ of $U'$ with $\rk U''=\rk U'-\rk U\geq \rk V-2 \rk U$.
    It now follows from the choice of $U$ that $h|_{U''}$ is anisotropic.
    Hence $\rk U''\leq u(D,\gamma,\varepsilon)$.
    We conclude that $2\rk U\geq \rk V-\rk U''\geq \rk h -u(D,\gamma,\varepsilon)$.
\end{proof}

The following statement is well-known. In \cite[Prop. 2.2]{Mah05}, it is shown for the case where $B$ is a division algebra. 
We include a proof that does not require this hypothesis.

\begin{prop}\label{3 types-hermitian u-invariant}
    Let $\gamma$ and $\gamma'$ involutions on $B$.
\begin{enumerate}[$(1)$]
    \item Let $\varepsilon\in\{\pm 1\}$. If $\gamma$ and $\gamma'$ are either both orthogonal or both symplectic, then $u(B,\gamma,\varepsilon)=u(B,\gamma',\varepsilon)$.
    If $\gamma$ is orthogonal and $\gamma'$ is symplectic, then $u(B,\gamma,\varepsilon)=u(B,\gamma',-\varepsilon)$.
    \item If $\gamma$ and $\gamma'$ are both unitary and $\gamma|_K=\gamma'|_K$, then $u(B,\gamma,1)=u(B,\gamma',-1)$.
\end{enumerate}	
\end{prop}
\begin{proof}
    In both parts, we have that $\gamma|_K=\gamma'|_K$. Hence $\gamma\circ\gamma'$ is a $K$-automorphism of $B$. 
    We obtain by the Skolem-Noether Theorem that $\gamma\circ\gamma'=\Int(b)$ for some $b\in\mg{B}$. Then $\gamma=\Int(b)\circ \gamma'$.
    
    $(1)$ If $\gamma$ and $\gamma'$ are either both orthogonal or both symplectic, then we set $\varepsilon'=1$,  otherwise we set $\varepsilon'=-1$.
    Then $\gamma'(b)=\varepsilon'b$, by \cite[Prop. 2.7]{KMRT98}.
    Let $V$ be a finitely generated $B$-right module and $h:V\times V\to B$ an $\varepsilon$-hermitian form over $(B,\gamma)$. 
    One easily verifies that $b^{-1}h:V\times V\to B$ is an $\varepsilon\varepsilon'$-hermitian form over $(B,\gamma')$. 
    Clearly, $h$ is isotropic if and only if $b^{-1}h$ is isotropic, and $\rk h=\rk b^{-1}h$.  
    Similarly, if $h':V\times V\to B$ is an $\varepsilon\varepsilon'$-hermitian form over $(B,\gamma')$, then $bh:V\times V\to B$ is an $\varepsilon$-hermitian form over $(B,\gamma)$. 
    This shows that $u(B,\gamma',\varepsilon\varepsilon')=u(B,\gamma,\varepsilon)$. 
	
    $(2)$ Set $\lambda=b\gamma'(b)^{-1}$. We obtain that $\Int(b)\circ\gamma'\circ\Int(b)=\Int(\lambda)\circ \gamma'$, hence $\id_B=\gamma\circ\gamma=\Int(\lambda)$, whereby $\lambda\in\mg{B}\cap \Z(B)=\mg{K}$.
    In particular $-\lambda^{-1}\gamma'(-\lambda^{-1})=1$. 
    By Hilbert's Theorem 90 applied to the quadratic extension $K/F$ where $F=\{x\in K\mid \gamma'(x)=x\}$, there exists $\lambda'\in \mg{K}$ such that $\lambda'^{-1}\gamma'(\lambda')=-\lambda^{-1}$.
    Set $b'=\lambda'b$. Then $\Int(b')=\Int(b)$ and thus $\gamma=\Int(b')\circ \gamma'$.
    Consider a finitely generated $B$-right module $V$. 
    Given a hermitian form $h:V\times V\to B$ over $(B,\gamma)$, we obtain that $b'^{-1}h$ is a  skew-hermitian form over $(B,\gamma')$. 
    Clearly $h$ is isotropic if and only if $b'^{-1}h$ is isotropic, and $\rk h=\rk b'^{-1}h$.  
    Similarly, if $h':V\times V\to B$ is a skew-hermitian form over $(B,\gamma')$, then $b'h$ is a hermitian form over $(B,\gamma)$. 
    This shows that $u(B,\gamma',-1)= u(B,\gamma,1)$.
\end{proof}

By means of \Cref{3 types-hermitian u-invariant}, we can reduce the number of different hermitian $u$-invariants to be considered for a central simple $F$-algebra.

Assume for now that $\cchar F\neq2$.
Given a central simple  $F$-algebra $B$ with $\exp B\leq2$, we fix an arbitrary orthogonal involution on $B$, which exists by \Cref{T:1st-ex-Albert}, and we set
\begin{equation*}
    u^+(B)=u(B,\gamma,1)\quad \text{and}\quad u^-(B)=u(B,\gamma,-1),
\end{equation*}
observing that, by \Cref{3 types-hermitian u-invariant}~$(1)$, the definition does not depend on the particular choice of $\gamma$.
For a given central simple $F$-algebra $B$ with $\exp B>2$, we set $u^+(B)=u^-(B)=0$.
We call $u^+(B)$ the \emph{orthogonal $u$-invariant of $B$} and $u^-(B)$ the \emph{symplectic $u$-invariant of $B$}.
These notations go back to \cite[Remark 2.3]{Mah05}. 
Note that $u^-(B)=0$ if $B$ is split.

Let us point out the relation between the orthogonal $u$-invariant and the classical $u$-invariant of a field.
Recall that the field $F$ is \emph{nonreal} if $-1$ is a sum of squares in $F$ and \emph{real} otherwise.
If $F$ is nonreal, one defines 
$$u(F) = \sup\{\dim(q)\mid\text{$q$ anisotropic quadratic form over $F$}\}\,. $$ 
This is called the \emph{$u$-invariant of $F$}. We refer to \cite[Chap.~8]{Pfi95} for a treatment of the $u$-invariant, including a discussion of how to extend this notion to cover real fields, which however is not relevant for this article.

\begin{ex}
Assume that $\cchar F\neq 2$. We consider $u^+(B)$ for the central simple $F$-algebra $B=F$. 
By the $1$-$1$-correspondence between symmetric bilinear forms and quadratic forms over $F$, we obtain that $u^+(F)=u(F)$.
\end{ex}

We now define the hermitian $u$-invariant for unitary involutions in a similar way.
Here we make no assumption on the characteristic of $F$.
Consider a separable quadratic field extension $K/F$ and a central simple $K$-algebra $B$.
If there exists a $K/F$-unitary involution $\gamma$ on $B$, (which by \Cref{exist-inv} is if and only if the corestriction algebra $\corr_{K/F}B$ is split), then we set 
$$u(B/F)=u(B,\gamma,1),$$
observing that this is independent of the specific choice of the $K/F$-unitary involution $\gamma$, in view of \Cref{3 types-hermitian u-invariant}~$(2)$.
If $B$ does not admit any $K/F$-unitary involution, then we set $u(B/F)=0$. 
We call $u(B/F)$ the \emph{$F$-unitary $u$-invariant of $B$}.

Note that the hermitian $u$-invariants depend only on the Brauer class of the algebra. This was pointed out in \cite[Lemma 2.1]{Wu18}.

\begin{prop}
    Let $K/F$ be a separable field extension with $[K:F]\leq2$. Let $B$ and $B'$ be central simple $K$-algebras such that $B\sim B'$. 
    Then we have $u^+(B)=u^+(B')$, $u^-(B)=u^-(B')$ and $u(B/F)=u(B'/F)$.
\end{prop}
\begin{proof}
    This follows by \cite[Theorem 9.3.5]{Knu91} together with \Cref{correspondence-(skew)hermit-involution}.
\end{proof}

Let $\Br(F)$ denote the Brauer group of $F$ and let $\Br_2(F)$ denote its $2$-torsion subgroup.
For $\alpha\in\Br(F)$, we take the central $F$-division algebra $D$ such that $\alpha=[D]$ and define
$$u^+(\alpha)=u^+(D)\qquad\mbox{ and }\qquad u^-(\alpha)=u^-(D)\,.$$
Similarly, given a separable quadratic field extension $K/F$ and $\alpha\in\Br(K)$, we take the central $K$-division algebra $D$ with $\alpha=[D]$
and define $$u(\alpha/F)=u(D/F)\,.$$

\begin{prop}\label{deg-ind-hermit-u-inv}
    Let $(A,\s)$ be an $F$-algebra with involution. 
    For $(a)$ and $(b)$, assume that $\cchar F\neq 2$.
\begin{enumerate}[$(a)$]
    \item If $\s$ is orthogonal and $\deg A> \ind A\cdot u^+(A)$, then $\s$ is isotropic.
    \item If $\s$ is symplectic and $\deg A> \ind A\cdot u^-(A)$, then $\s$ is isotropic.
    \item If $\s$ is unitary and $\deg A>\ind A\cdot u(A/F)$, then $\s$ is isotropic.
\end{enumerate}
\end{prop}
\begin{proof}
    Suppose that $\s$ is anisotropic. Then $K=\Z(A)$ is a field.
    Let $D$ be the central $K$-division algebra Brauer equivalent to $A$. Then $\deg D=\ind A$.
    We fix an involution $\gamma$ on $D$ with $\gamma|_K=\s|_K$ and such that $\gamma$ is orthogonal if and only if $\s$ is orthogonal. 
    By \Cref{correspondence-(skew)hermit-involution}, $(A,\s)\simeq\Ad_D(h)$ for some nonsingular hermitian form over $(D,\gamma)$. 
    Then $\deg A= \ind A\cdot\rk h$.
    
    Since $\s$ is anisotropic, by \Cref{isotropic-hyperbolic-inv-equiv-cond}, $h$ is anisotropic. 
    If $\s$ is orthogonal, then so is $\gamma$, and it follows that $\rk h\leq u^+(D)=u^+(A)$.
    If $\s$ is symplectic, then so is $\gamma$, and we obtain that $\rk h\leq u^-(D)=u^-(A)$.
    Finally, if $\s$ is unitary, then as $\gamma|_K=\s|_K$, we conclude that $\rk h\leq u(D/F)=u(A/F)$.
\end{proof}

Within the study of the $u$-invariant for quadratic forms, a central topic is its behavior under field extensions.

\begin{prop}\label{u-inv-behavior-2-ext}
    Assume that $F$ is nonreal and let $L/F$ be a finite $2$-extension. For $n\in\nat$ with $[L:F]=2^n$, we have $u(L)\leq(\frac{3}{2})^{n}\cdot u(F)$.
\end{prop}
\begin{proof} 
    See \cite[Chap.~9, Cor.~2.2]{Pfi95} for the case where $n=1$.
    From this case, the statement follows by induction on $n$.
\end{proof}

We recall the following fact about the unitary $u$-invariant in the split case.
\begin{prop}\label{u-inv-unitary-split-bounds}
    Assume that $F$ is nonreal. Let $K/F$ be a separable quadratic field extension. The following hold:
\begin{enumerate}[$(a)$]
    \item $u(K/F)\leq\frac{1}{2}u(F)$.
    \item 
    If $u(K/F)>2$, then there exists an anisotropic quadratic $3$-fold Pfister form over $F$.
\end{enumerate}	
\end{prop}
\begin{proof}
    We recall from \cite[Chap.~10, \S 1]{Scha85} that, for an arbitrary nonsingular hermitian form over $(K,\can_{K/F})$, the rule $x\mapsto h(x,x)$ defines a quadratic form $\q_h$ over $F$ with $\dim\q_h=2\rk h$  which becomes hyperbolic over $K$, and that $h$ is isotropic if and only if $\q_h$ is isotropic.
    
    $(a)$ As $u(K/F)=u(K,\can_{K/F},1)$, this  follows directly.
	
    $(b)$ Assume that $u(K/F)>2$. We may then choose an anisotropic hermitian form $h$ over $(K/F,\can_{K/F})$ with $\rk h=3$. 
    Then $\q_h$ is a $6$-dimensional quadratic form over $F$ which is similar to a subform of some quadratic $3$-fold Pfister form $\rho$ over $F$.
    Since $\q_h$ is anisotropic, $\rho$ is not hyperbolic.
    Since $\rho$ is a Pfister form, it follows that $\rho$ is anisotropic.
\end{proof}

Using systems of quadratic forms, Mahmoudi obtained in \cite{Mah05} the following upper bounds on the hermitian $u$-invariants in terms of the $u$-invariant of the base field, which we restate here in our setup.

\begin{prop}[Mahmoudi]\label{Mahmoudi-hemit-sym/skew+u-inv}
    Assume that $\cchar F\neq 2$ and $F$ is nonreal.
\begin{enumerate}[$(a)$]
    \item Let $\alpha\in\Br(F)$ and $n=\ind\alpha$. Then 
\begin{equation*}
    \mbox{$u^+(\alpha)\leq\frac{(n+1)(n^2+n+2)}{8n}\cdot u(F)$}\quad\text{and}\quad\mbox{$u^-(\alpha)\leq\frac{(n-1)(n^2-n+2)}{8n}\cdot u(F)$}.
\end{equation*}
    \item Let $K/F$ be a quadratic field extension, $\alpha\in\Br(K)$ and $n=\ind\alpha$. Then 
\begin{equation*}
    \mbox{$u(\alpha/F)\leq\frac{n^2+1}{4}\cdot u(F)$.}
\end{equation*}
\end{enumerate}	 	
\end{prop}
\begin{proof}
    See \cite[Prop. 3.6]{Mah05}.
\end{proof}

The hermitian $u$-invariants were further studied in \cite{PS13} and \cite{Wu18}.

\begin{thm}[Parihar-Suresh]\label{T:PH}
    Assume that $F$ is nonreal with $\cchar F\neq 2$. 
    Let $K/F$ be a quadratic field extension and $\alpha\in\Br(F)$. 
    Then $u^+(\alpha_K)\leq \frac{3}{2}u^+(\alpha)$, $u^{-}(\alpha_K)\leq \frac{3}2 u^{-}(\alpha)$ and $u(\alpha_K/F)\leq \frac{1}2 u^+(\alpha)+u^-(\alpha)$.
    \end{thm}
\begin{proof}
    See \cite[Theorems~4.2 \& 4.3]{PS13}.
\end{proof}

Using \Cref{T:PH}, the following bounds were obtained in \cite{Wu18}.

\begin{thm}[Wu]\label{T:Wu}
    Assume that $F$ is nonreal and $\cchar F\neq2$. 
    Let $\alpha\in\Br(F)$. Assume that $n\in\nat$ is such that $\alpha$ is given by a tensor product of $n$\, $F$-quaternion algebras.
    Then the following hold:
\begin{enumerate}[$(a)$]
    \item \mbox{$u^+(\alpha)\leq\frac{1}{5}(4+(\frac{3}{2})^{2n})u(F)$}.
    \item \mbox{$u^-(\alpha)\leq\frac{1}{5}(-1+(\frac{3}{2})^{2n})u(F)$}.
    \item \mbox{$u(\alpha_K/F)\leq\frac{1}{5}(1+(\frac{3}{2})^{2n+1})u(F)$} for any quadratic field extension $K/F$.
\end{enumerate}	
\end{thm}
\begin{proof}
    This is \cite[Theorem 1.3]{Wu18}. The proof is by induction on $n$.
    For $n=1$, $(a)$, $(b)$, and $(c)$ are established in \cite[Cor. 3.4]{Mah05}, \cite[Chap. 10, Theorem 1.7]{Scha85}, and \cite[Cor. 4.4]{PS13}, respectively.
\end{proof}

\section{Isotropy via quadratic reduction}\label{section:study-isotropy-via-quad-reduc}

Given a central simple algebra with involution and suitable separable quadratic field extension contained in it, one obtains a simple subalgebra carrying a pair of involutions.
Accordingly, we can obtain from a hermitian form over this algebra with involution a pair of hermitian forms over the subalgebra with respect to two different involutions such that the isotropy of this pair is equivalent to the isotropy of the original hermitian form. 
By these means, one can compare the hermitian $u$-invariant of a central simple algebra with involution with the hermitian $u$-invariants of a subalgebra.
This method to study the hermitian $u$-invariants stems from  \cite[Prop. 3.1]{Mah05}. 
The purpose of this section is to extend this method by relaxing the hypotheses under which it can be applied.
This will be achieved with \Cref{T:main}.
Our treatment will cover the case of characteristic $2$ for the unitary case.
As described above, the method relies on the presence of a quadratic subextension. 
In \Cref{unitary&orthogonal-hermit-sep-quad} we discuss  another technique to deal with the complementary situation where we cannot find any quadratic field extension of the center contained in the algebra. 
Together, these methods will allow us in the subsequent sections to obtain bounds on the hermitian $u$-invariants of a central simple algebra with involution in terms of the degree and the $u$-invariant of a splitting $2$-extension; see \Cref{2-ext-unitary-hermit} and \Cref{orthogonal-hermit-split-multi}. 

\medskip
In this section, let $K$ be a field.

\begin{prop}\label{quad-ext-excellent-inv-division}
    Let $D$ be a central $K$-division algebra and $\gamma$ an involution on $D$.
    Assume that $\cchar K\neq2$ if $\gamma|_K=\id_K$.
    Let $\varepsilon\in\{\pm1\}$ and $h$ be a nonsingular $\varepsilon$-hermitian form over $(D,\gamma)$. 
    Let $M/K$ be a separable quadratic field extension such that 
    $D_{M}$ is a division algebra. Then 
\begin{eqnarray*}
    h\simeq h'\perp h_1\perp\ldots\perp h_n
\end{eqnarray*}
    for some $n\in\nat$ and $\varepsilon$-hermitian forms $h',h_1,\ldots,h_n$ over $(D,\gamma)$ such that $h'_{M}$ is anisotropic and, for $1\leq i\leq n$, we have that $\rk h_i=2$ and $(h_{i})_M$ is hyperbolic.
\end{prop}
\begin{proof}
    See \cite[Prop. 4.3]{BB24} and its proof. (In \cite{BB24}, the general assumption is that $\gamma|_K=\id_K$ and $\cchar K\neq 2$, but the proof \cite[Prop. 4.3]{BB24} is valid in general.)
\end{proof}

The following proposition extends \cite[Prop.~3.1]{Mah05}. (Here the assumptions on $\alpha$ are less restrictive, and we cover characteristic $2$ for the unitary case.)
This extension is complementary to those obtained in \cite[\S 4]{PS13}, while based on similar arguments.

\begin{prop}\label{unitary&orthogonal-hermit-sep-quad}
    Let $K/F$ be a separable field extension with $[K:F]\leq 2$. 
    Let $M/K$ be a quadratic field extension such that $M/F$ is separable and $M$ contains a quadratic extension of $F$ distinct from $K$. 
    Let $\alpha\in\Br(K)$ be such that $\ind \alpha_M=\ind \alpha$. Then the following hold:
\begin{enumerate}[$(a)$]
    \item If $\cchar F\neq 2$ and $K=F$, then 
\begin{equation*}
    u^+(\alpha)\leq u^+(\alpha_M)+2u(\alpha_M/F)\quad\text{and}\quad u^-(\alpha)\leq u^-(\alpha_M)+2u(\alpha_M/F).
\end{equation*}	
    \item If $[K:F]=2$, then $M/F$ has a quadratic subextension $L/F$ such that $M=KL$ and 
\begin{equation*}
    \mbox{$u(\alpha/F)\leq3\cdot u(\alpha_{M}/L)$.}
\end{equation*}	
\end{enumerate}
\end{prop}
\begin{proof}
    If $K=F$, then we assume that $\cchar F\neq 2$. Let $D$ be the central $K$-division algebra such that $\alpha=[D]$.
    If $K=F$ and $u^+(\alpha)=0$, respectively if $[K:F]=2$ and $u(\alpha/F)=0$, then the claim of $(a)$, respectively of $(b)$ holds trivially.
    We may assume that this is not the case.
    We may thus fix an $F$-linear involution $\gamma$ on $D$ with $F=\{x\in K\mid \gamma(x)=x\}$ which is orthogonal or unitary (depending on $[K:F]$).
    
    Note that $D_{M}$ is a division algebra because $\ind\alpha_{M}=\ind\alpha$, and $\gamma_M=\gamma\otimes \id_M$ is an $F$-linear involution on $D_{M}$.

    Let $\varepsilon\in\{\pm 1\}$ and consider an anisotropic $\varepsilon$-hermitian form $h$ over $(D,\gamma)$, defined on some finite-dimensional $D$-right vector space $V$. 
    By \Cref{{quad-ext-excellent-inv-division}}, $h\simeq h'\perp h''$ for some $\varepsilon$-hermitian forms $h'$ and $h''$ such that $h'_M$ is anisotropic and $h''_M$ is hyperbolic. 
    In particular, $\rk h'=\rk h'_M\leq u(D_M,\gamma_M,\varepsilon)$.
    
    Set $(A,\sigma)=\Ad_D(h'')$. Then $(A,\sigma)$ is an $F$-algebra with involution and $\deg A=\rk h''\cdot\deg D$. 
    Since $h$ is anisotropic, so is $h''$, and hence $\sigma$ is anisotropic.
    As $h''_M$ is hyperbolic and $\Ad_{D_M}(h''_M)\simeq (A_M,\sigma_M)$, we obtain $\sigma_M$ is hyperbolic.
    It follows by \cite[Theorem 1.15 and Theorem 1.16]{BFT07} that we can identify $M$ with an $F$-subalgebra of $A$ with $\sigma(M)=M$ and $\sigma|_M\neq \id_M$.

    Set $C=\C_A(M)$. Then $C$ is a central simple $M$-algebra and $\sigma|_C$ is an anisotropic unitary involution on $C$.
    Since $[C]=[D_M]=\alpha_M$ in $\Br(M)$, we can choose an involution $\tilde{\gamma}$ on $D_M$ with $\tilde{\gamma}|_M=\sigma|_M$.
    By \Cref{correspondence-(skew)hermit-involution}, it follows that $(C,\sigma|_C)\simeq\Ad_{D_{M}}(\tilde{h})$ for some anisotropic hermitian form $\tilde{h}$ over $(D_{M},\tilde{\gamma})$.
    Hence $\rk \tilde h\leq u(D_M,\tilde{\gamma},1)$.
    Since $\deg D=\ind \alpha=\ind\alpha_M=\deg D_M$ and 
    $\deg D\cdot\rk h''=\deg A= 2\deg C=2\deg D_M \cdot \rk \tilde{h}$, we now conclude that $\rk h''=2\rk \tilde{h}\leq 2 u(D_M,\tilde{\gamma},1)$.

    It follows that $\rk h=\rk h'+\rk h''\leq u(D_M,\gamma_M,\varepsilon)+2u(D_M,\tilde{\gamma},1)$.

    This argument shows that 
    $$u(D,\gamma,\varepsilon)\leq u(D_M,\gamma_M,\varepsilon)+2u(D_M,\tilde{\gamma},1)\,.$$

    If $\cchar F\neq 2$ and $K=F$, then using the two possible choices of $\varepsilon\in\{\pm 1\}$, this proves the two inequalities in $(a)$.

    Assume now that $[K:F]=2$. We take $\varepsilon=1$.
    Set $\gamma_0=\gamma_M$ and $\gamma_1=\tilde{\gamma}$, and for $i=0,1$ let $L_i=\{x\in M\mid \gamma_i(x)=x\}$.
    Then the above inequality yields that $u(\alpha/F)=u(D,\gamma,1)\leq u(D_M,\gamma_0,1)+2u(D_M,\gamma_1,1)=u(\alpha_M/L_0)+2u(\alpha_M/L_1)$.
    Hence letting $k\in\{0,1\}$ be such that $u(\alpha_M/L_k)\geq u(\alpha_M/L_{1-k})$, we conclude that $u(\alpha/F)\leq 3u(\alpha_M/L_k)$, which establishes $(b)$.
\end{proof}

\begin{prop}\label{sep-quad-fromF-stable-under-inv}
    Let $D$ be a central $K$-division algebra and let $L/F$ be a separable quadratic field extension contained in $D$ and linearly disjoint from $K/F$.
    There exists $j\in D^{\times}$ such that $\Int(j)|_L=\can_{L/F}$.
    Furthermore, given such $j$ and an involution $\gamma$ on $D$ with $F=\{x\in K\mid \gamma(x)=x\}$, there exists an involution $\gamma'$ on $D$ such that $\gamma'(j)=-j$ and $\gamma'|_K=\gamma|_K$, $\gamma'|_L=\can_{L/F}$.
\end{prop}
\begin{proof}
    By the Skolem-Noether Theorem, there exists an element $j\in D^{\times}$ such that $\Int(j)|_{KL}=\can_{KL/K}$. Then $\Int(j)|_L=\can_{L/F}$. 
    Consider an involution $\gamma$ on $D$ with $F=\{x\in K\mid \gamma(x)=x\}$.  	
    Note that $L(j^2)$ is a subfield of $D$.
    As $\Int(j)|_L=\can_{L/F}\neq \id_L$ we have $j\notin L(j^2)$. Set $H=L(j^2)\oplus jL(j^2)$ and $Q=KL(j^2)\oplus jKL(j^2)$. 
    Then $H$ is an $F(j^2)$-quaternion algebra and $Q$ is a $K(j^2)$-quaternion algebra contained in $D$.
    Note that $H$ and $K(j^2)$ are $F(j^2)$-subalgebras of $Q$, and $j^2$ commutes with $KL$ and hence with $Q$.
    Multiplication in $Q$ induces an isomorphism of $F(j^2)$-algebras $H\otimes_{F(j^2)}K(j^2)\to Q$.
    Let $\tau$ be the involution on $Q$ corresponding with $\can_H\otimes\can_{K(j^2)/F(j^2)}$ under this isomorphism.  
    Then $\tau|_K=\can_{K/F}=\gamma|_K$.
    Hence, by \cite[Theorem 4.14]{KMRT98}, there exists an involution $\gamma'$ on $D$ such that $\gamma'|_Q=\tau$. 
    In particular, $\gamma'|_K=\gamma|_K$ and $\gamma'|_L=\can_{L/F}$.
\end{proof}

\begin{prop}\label{properties of projections}
    Let $D$ be a central $K$-division algebra and $\gamma$ an involution on $D$. Let $F=\{x\in K\mid \gamma(x)=x\}$.
    Let $L/F$ be a separable quadratic field extension contained in $D$, linearly disjoint from $K/F$ and with $\gamma|_L=\can_{L/F}$.
    Let $j\in D^{\times}$ be such that $\gamma(j)=-j$ and $\Int(j)|_L=\can_{L/F}$. 
    Set $\tilde{D}=C_D(L)$, $\gamma_0=\gamma|_{\tilde{D}}$ and $\gamma_1=(\Int(j^{-1})\circ\gamma)|_{\tilde{D}}$. 
    Then $\tilde{D}$ is an $L$-division algebra with $\Z(\tilde{D})=LK$ and such that $D=\tilde{D}\oplus j\tilde{D}$.
    Furthermore, the following hold:
\begin{enumerate}[$(a)$]
    \item The maps $\gamma_0$ and $\gamma_1$ are involutions on $\tilde{D}$, and $\gamma_0$ is unitary.
    If $\gamma$ is unitary, then so is~$\gamma_1$.
    If $\cchar F\neq 2$, then $\gamma$ is orthogonal if and only if $\gamma_1$ is symplectic, and vice-versa.
    \item Let $\pi_0,\pi_1:D\to\tilde{D}$ be the $F$-linear maps such that $\id_D=\pi_0+j\pi_1$.
    Let $\varepsilon\in\{\pm1\}$ and let $h:V\times V\to D$ be a nonsingular $\varepsilon$-hermitian form over $(D,\gamma)$ defined on a finite-dimensional $D$-right vector space $V$. 
    Then $\pi_0\circ h:V\times V\to\tilde{D}$ is a nonsingular $\varepsilon$-hermitian form over $(\tilde{D},\gamma_0)$ and $\pi_1\circ h:V\times V\to\tilde{D}$ is a nonsingular $(-\varepsilon)$-hermitian form over $(\tilde{D},\gamma_1)$.
\end{enumerate}	
\end{prop}
\begin{proof}
    It is easy to see that $\gamma(\tilde{D})=\tilde{D}$, $\Z(\tilde{D})=KL$, $D=\tilde{D}\oplus j\tilde{D}$ and that $\gamma_0$ and $\gamma_1$ are involutions on $\tilde{D}$.

    $(a)$ Clearly $(\gamma_0)|_L\neq\id_L$, so $\gamma_0$ is unitary. Note that $(\gamma_1)|_L=\id_L$ and $(\gamma_1)|_K=\gamma|_K$. 
    Therefore $\gamma_1$ is unitary if and only if $\gamma$ is unitary. Assume now that this is not the case, that is $\gamma|_K=\id_K$ and $F=K$.
    Since $\gamma(j)=-j$, assuming that $\cchar F\neq2$, it follows by \cite[Prop. 2.7]{KMRT98} that if $\gamma$ is orthogonal, then $\gamma_1$ is symplectic, and vice-versa.

    $(b)$ Set $h_0=\pi_0\circ h$ and $h_1=\pi_1\circ h$. Clearly the maps $h_0$ and $h_1$ are bi-additive, and for any $v,w\in V$ we have $h(v,w)=h_0(v,w)+jh_1(v,w)$.
    Let $v,w\in V$ and $x,y\in\tilde{D}$. We have
\begin{equation*}
\begin{split}
    h(vx,wy)=\gamma(x)h(v,w)y
    &=\gamma(x)h_0(v,w)y+\gamma(x)jh_1(v,w)y\\
    &=\gamma_0(x)h_0(v,w)y+j\gamma_1(x)h_1(v,w)y,
\end{split}
\end{equation*}
   whereby $h_0(vx,wy)=\gamma_0(x)h_0(v,w)y$ and $h_1(vx,wy)=\gamma_1(x)h_1(v,w)y$. Furthermore, 
\begin{equation*}
\begin{split}
    h(w,v)=\varepsilon\gamma(h(v,w))&=\varepsilon\gamma(h_0(v,w))-\varepsilon\gamma(h_1(v,w))j\\
    &=\varepsilon\gamma_0(h_0(v,w))-\varepsilon j\gamma_1(h_1(v,w)).
\end{split}
\end{equation*}  
   whereby $h_0(w,v)=\varepsilon\gamma_0(h_0(v,w))$ and $h(w,v)=-\varepsilon\gamma_1(h_1(v,w))$. 
   Hence, $h_0$ is an $\varepsilon$-hermitian form over $(\tilde{D},\gamma_0)$ and $h_1$ is a $(-\varepsilon)$-hermitian form over $(\tilde{D},\gamma_1)$.

   Note that for $v,w\in V$, since $\gamma(j)=-j$, we have
\begin{equation*}
\begin{split}
    h(vj,w)&=h_0(vj,w)+jh_1(vj,w)\\
    &=-jh_0(v,w)-j^2h_1(v,w),
\end{split}
\end{equation*}  
   and as $j^2\in\mg{\tilde{D}}$, we get that $h_0(vj,w)=-j^2h_1(v,w)$ and $h_1(vj,w)=-h_0(v,w)$. 
   As $h$ is nonsingular, this implies in particular that $h_0$ and $h_1$ are nonsingular.
\end{proof}

We obtain another generalization of \cite[Prop. 3.1]{Mah05}.

\begin{prop}\label{unitary&orthogonal-hermit-sep-quad-ind-red}
    Let $K/F$ be a separable field extension with $[K:F]\leq 2$. 
    Let $M/K$ be a quadratic field extension such that $M/F$ is separable and $M$ contains a quadratic extension of $F$ distinct from $K$. 
    Let $\alpha\in\Br(K)$ be such that $\ind \alpha_M=\frac{1}2\ind \alpha$. Then the following hold:
\begin{enumerate}[$(a)$]
    \item If $\cchar F\neq 2$ and $K=F$, then 
\begin{equation*}
    \mbox{$u^+(\alpha)\leq\frac{1}{2}u^+(\alpha_M)+u(\alpha_M/F)\quad\text{and}\quad u^-(\alpha)\leq u^-(\alpha_M)+\frac{1}{2}u(\alpha_M/F)$}.
\end{equation*}	
    \item If $[K:F]=2$, then $M/F$ has a quadratic subextension $L/F$ such that $M=KL$ and 
\begin{equation*}
    \mbox{$u(\alpha/F)\leq\frac{3}{2}\cdot u(\alpha_{M}/L)$.}
\end{equation*}	
\end{enumerate}
\end{prop}
\begin{proof}
    If $K=F$, then we assume that $\cchar F\neq 2$. Let $D$ be the central $K$-division algebra such that $\alpha=[D]$.
    If $K=F$ and $u^+(\alpha)=0$, respectively if $[K:F]=2$ and $u(\alpha/F)=0$, then the claim of $(a)$, respectively of $(b)$ holds trivially.
    We may assume that this is not the case.
    We may thus fix an $F$-linear involution $\gamma$ on $D$ with $F=\{x\in K\mid \gamma(x)=x\}$ which is orthogonal or unitary (depending on $[K:F]$).

    Since $\ind\alpha_{M}=\frac{1}{2}\ind\alpha$, we can embed $M$ into $D$ and hence view $M$ as an $F$-subalgebra of $D$. 
    By means of \Cref{sep-quad-fromF-stable-under-inv}, we may fix $j\in \mg{D}$ and an $F$-linear involution $\gamma$ on $D$ such that $\gamma(j)=-j$ and $\Int(j)|_M=\can_{M/F}=\gamma|_M$.

    Set $\tilde{D}=C_D(M)$, and $\gamma_0=\gamma|_{\tilde{D}}$, $\gamma_1=(\Int(j^{-1})\circ\gamma)|_{\tilde{D}}$.
    Then $C(\tilde{D})=M$, $\gamma_0$ and $\gamma_1$ are $F$-linear involutions on $\tilde{D}$, whose properties are, according to  \Cref{properties of projections}, determined by those of $\gamma$ in the following way:
    $\gamma_0$ is unitary in any case, and if $\gamma$ is orthogonal, then $\gamma_1$ is symplectic, whereas if $\gamma$ is unitary, then so is $\gamma_1$.

    Let $\varepsilon\in\{\pm 1\}$ and consider an $\varepsilon$-hermitian form $h$ over $(D,\gamma)$, defined on some finite-dimensional $D$-right vector space $V$.
    Seeing $V$ as a $\tilde{D}$-right vector space, \Cref{properties of projections} yields an $\varepsilon$-hermitian form $h_0$ over $(\tilde{D},\gamma_0)$ and a $(-\varepsilon)$-hermitian form $h_1$ over $(\tilde{D},\gamma_1)$ such that $\rk h_0=\rk h_1=2\cdot\rk h$.

    Suppose now that $\rk h> \frac{1}2 u(\tilde{D},\gamma_0,\varepsilon)+u(\tilde{D},\gamma_1,-\varepsilon)$.
    As $\rk h_0=2\rk h$, it follows by \Cref{L:ttis-subspace-u} that $V$ contains a $\tilde{D}$-subspace $W$  such that $h_0(v,w)=0$ for all $v,w\in W$ and $\rk_{\tilde{D}} W>u(\tilde{D},\gamma_1,-\varepsilon)$. 
    This implies that $h_1(v,v)=0$ for some $v\in W\setminus\{0\}$. Then $h(v,v)=h_0(v,v)+jh_1(v,v)=0$, so $h$ is isotropic.
    This argument proves that 
\begin{eqnarray*}
    u(D,\gamma,\varepsilon)\leq \mbox{$\frac{1}2$} u(\tilde{D},\gamma_0,\varepsilon)+u(\tilde{D},\gamma_1,-\varepsilon)\,.
\end{eqnarray*}
    Analogously, switching the roles of $h_0$ and $h_1$ in the argument, we obtain that
\begin{eqnarray*}
    u(D,\gamma,\varepsilon)\leq u(\tilde{D},\gamma_0,\varepsilon)+\mbox{$\frac{1}2$} u(\tilde{D},\gamma_1,-\varepsilon)\,.
\end{eqnarray*}

    We are now ready to prove the two parts of the statement.
    If $K=F$, then the first inequality for $\varepsilon=-1$ together with the second inequality for $\varepsilon=1$ yield the two inequalities in $(a)$.
    
    Assume now that $[K:F]=2$. For $i=0,1$, we set $L_i=\{x\in M\mid \gamma_i(x)=x\}$.
    Then $K/F$, $L_0/F$ and $L_1/F$ are the three different quadratic subextensions in $M/F$. 
    Note that $u(\tilde{D},\gamma_i,\varepsilon)=u(\tilde{D},\gamma_i,1)$ for $i=0,1$.
    Take $k\in \{0,1\}$ such that $u(\tilde{D},\gamma_k,1)\geq u(\tilde{D},\gamma_{1-k},1)$. 
    We conclude from either of the two inequalities above that $u(\alpha/F)=u(D,\gamma,1)\leq \mbox{$\frac{3}2$} u(\tilde{D},\gamma_k,1)=\mbox{$\frac{3}2$} u(\alpha_M/L_{k})$, proving $(b)$. 
\end{proof} 

We merge the main results of this section into one theorem.

\begin{thm}
    \label{T:main}
    Let $K/F$ be a separable field extension with $[K:F]\leq 2$. 
    Let $M/K$ be a quadratic field extension such that $M/F$ is separable and $M$ contains a quadratic extension of $F$ distinct from $K$. 
    Let $\alpha\in\Br(K)$. 
\begin{enumerate}[$(a)$]
    \item If $\cchar F\neq 2$ and $K=F$, then 
\begin{eqnarray*}
    u^+(\alpha) & \leq &  \mbox{$\frac{\ind \alpha_{M}}{\ind \alpha}$}  \left(u^+(\alpha_M)+2u(\alpha_M/F)\right)\,.
\end{eqnarray*}	
    \item If $[K:F]=2$, then $M/F$ has a quadratic subextension $L/F$ such that $M=KL$ and 
\begin{eqnarray*}
    u(\alpha/F) & \leq &  3\cdot \mbox{$\frac{\ind\alpha_{KM}}{\ind \alpha}$} \cdot  u(\alpha_{M}/L).
\end{eqnarray*}	
\end{enumerate}
\end{thm}
\begin{proof}
    This follows from \Cref{unitary&orthogonal-hermit-sep-quad} if $\ind \alpha_{KM}=\ind \alpha$ and from \Cref{unitary&orthogonal-hermit-sep-quad-ind-red} if $\ind \alpha_{KM}=\frac{1}2\ind\alpha$.
\end{proof}

In \Cref{section: bounds on the unitary u-invariant} and \Cref{section: bounds on the orthogonal u-invariant}, we shall apply \Cref{T:main} to obtain bounds on the hermitian $u$-invariants.

\section{Bounds on the unitary $u$-invariant}\label{section: bounds on the unitary u-invariant}
To make use of \Cref{T:main} when given a $2$-extension splitting the algebra and aiming to bound the hermitian $u$-invariants, we should first address this problem for the unitary $u$-invariant. 
Therefore, in this section, we study the behaviour of the unitary $u$-invariant under $2$-extensions. 

\begin{thm}\label{2-ext-unitary-hermit-compare}
    Let $K/F$ be a separable quadratic field extension and \mbox{$\alpha\in\Br(K)$}. Let $n\in\nat$ and let $L/F$ be a $2$-extension linearly disjoint from $K/F$ such that $[L:F]=2^n$. 
    Then there exists a $2$-extension $L'/F$ linearly disjoint from $K/F$ with $KL'=KL$ and such that
\begin{equation*}
    \ind\alpha \cdot u(\alpha/F)\leq 3^n\cdot\ind\alpha_{KL}\cdot u(\alpha_{KL/L'}).
\end{equation*}	
\end{thm}
\begin{proof}
    We prove the statement by induction on $n$. If $n=0$, then $L=F$ and we can take $L'=F$. Assume now that $n\geq1$.
    Since $L/F$ is a $2$-extension, there exist a family of intermediate fields $(L_i)_{i=0}^n$ with $L_0=F$, $L_n=L$ and such that $L_i/L_{i-1}$ is a separable quadratic field extension for $1\leq i\leq n$. Set $K'=L_1K$. 
    
    By \Cref{T:main}, there exists a separable quadratic field extension $L'_1/F$ contained in $K'/F$ such that $L_1'K=K'$ and
\begin{equation*}
    \mbox{$u(\alpha/F)\leq\frac{\ind\alpha_{K'}}{\ind\alpha}\cdot3\cdot u(\alpha_{K'}/L'_1)$}.
\end{equation*}	 
    Since $[L'_1:F]=[K:F]=2$ and $[K':F]=4$, it follows that $L'_1/F$ is linearly disjoint from $K/F$.
    Note that there exists a $2$-extension $L'/L_1'$ contained in $LK$, linearly disjoint from $K'/L_1'$ and such that $L'K=LK$. 
    Moreover, for any such extension $L'/L_1'$ we have that $[L':L_1']=[LK:L_1K]=[L:L_1]=2^{n-1}$.
    The induction hypothesis yields that there exists such an extension $L'/L'_1$ with
\begin{equation*}
    \mbox{$u(\alpha_{K'}/L'_1)\leq\frac{\ind\alpha_{KL}}{\ind\alpha_{K'}}\cdot3^{n-1}\cdot u(\alpha_{KL/L'})$}.
\end{equation*} 
    Combining the inequalities yields that \mbox{$u(\alpha/F)\leq\frac{\ind\alpha_{KL}}{\ind\alpha}\cdot3^n\cdot u(\alpha_{KL/L'})$}.
\end{proof}

\begin{cor}\label{2-ext-unitary-hermit}
    Let $K/F$ be a separable quadratic field extension and  \hbox{$\alpha\!\in\Br(K)$}.
    Let $n\in\nat$ and let $L/F$ be a $2$-extension linearly disjoint from $K/F$ such that  $[L:F]=2^n$ and $\alpha_{KL}=0$.  
    Then, there exists a $2$-extension $L'/F$ linearly disjoint from $K/F$  such that $KL'=KL$ and
\begin{equation*}
    \ind\alpha\cdot u(\alpha/F)\leq 3^n\cdot u(KL/L').
\end{equation*}
\end{cor}
\begin{proof}
    By \Cref{2-ext-unitary-hermit-compare}, there exists a $2$-extension $L'/F$ linearly disjoint from $K/F$ with $KL'=KL$ and 
    $\ind\alpha\cdot u(\alpha/F)\leq\ind\alpha_{KL}\cdot 3^n\cdot u(\alpha_{KL/L'})$.	
    Since $\alpha_{KL}=0$, we have $\ind\alpha_{KL}=1$ and $u(\alpha_{KL/L'})=u(KL/L')$.
\end{proof}

We denote by $\Iq{3}F$ the subgroup of the Witt group of $F$ generated by the Witt equivalence classes of quadratic $3$-fold Pfister forms over $F$.
In particular, $\Iq{3}F=0$ if and only if every quadratic $3$-fold Pfister form over $F$ is hyperbolic.

\begin{cor}\label{unit-hermit-symbol-I_q^3=0}
    Assume that $\Iq{3}F=0$. Let $K/F$ be a separable quadratic field extension and  $\alpha\in\Br(K)$. 
    Let $n\in\nat$. Assume that there exists a $2$-extension $L/F$ linearly disjoint from $K/F$ with $[L:F]=2^n$ and such that $\alpha_L=0$. 
    Then 
\begin{equation*}
    \ind \alpha\cdot u(\alpha/F)\leq 2\cdot3^{n}.
\end{equation*}	
\end{cor}
\begin{proof}
    By \Cref{2-ext-unitary-hermit}, there exists a $2$-extension $L'/F$ linearly disjoint from $K/F$ with $KL'=KL$ such that $\ind \alpha \cdot u(\alpha/F)\leq 3^n\cdot u(KL/L')$.
    As $\Iq{3}F=0$ and $L'/F$ is a $2$-extension, it follows by a repeated use of \cite[Theorem 34.22]{EKM08} that $\Iq{3}L'=0$. Now \Cref{u-inv-unitary-split-bounds} yields that $u(KL/L')\leq2$.
\end{proof}

\begin{rem}
    Note that the bound in \Cref{unit-hermit-symbol-I_q^3=0} does not involve $u(F)$. 
    When $\ind\alpha=2$ and $\Iq{3}F=0$, we obtain that $u(\alpha/F)\leq 3$. An example showing that this bound is optimal will appear in a forthcoming article. 
\end{rem}

\begin{thm}\label{C:unitary-biquat}
    Assume that $F$ is nonreal. Let $K/F$ be a separable quadratic field extension.
    Let $\alpha\in\Br(K)$ be such that $\ind\alpha\leq 4$. Then $u(\alpha/F)\leq \frac{63}{32}u(F)$.
\end{thm}
\begin{proof}
    We may assume that $\ind \alpha=4$. It follows by \cite[Theorem 7.4]{BGBT18} that there exists a separable quadratic field extension $L/F$ linearly disjoint to $K/F$ such that $\ind\alpha_{KL}= 2$.
    Moreover, in view of \Cref{2-ext-unitary-hermit-compare}, we may choose $L/F$ in such way that $u(\alpha/F)\leq \frac{3}2 u(\alpha_{KL}/L)$.
    Now, $u(\alpha_{KL}/L)\leq \frac{7}8 u(L)$, by \Cref{T:Wu}.
    Moreover, $u(L)\leq \frac{3}{2}u(F)$, by \Cref{u-inv-behavior-2-ext}.
    Therefore $u(\alpha/F)\leq \frac{63}{32}u(F)$.
\end{proof}

\begin{rem}
    Let $K/F$ and $\alpha$ be as in \Cref{C:unitary-biquat}.
    If $\alpha=\beta_K$ for some $\beta\in\Br_2(F)$ with $\ind \beta\leq 4$, then \Cref{T:Wu} yields that $u(\alpha/F)\leq \frac{55}{32}u(F)$, which is better than the bound in \Cref{C:unitary-biquat}.
    If $\exp\alpha=2$ and $\ind \alpha= 4$, then can one derive from \Cref{T:Wu} that $u(\alpha/F)\leq \frac{463}{128}u(F)$, by using that $\alpha=(\gamma_1+\gamma_2+\gamma_3)_K$ for $\gamma_1,\gamma_2,\gamma_3\in\Br(F)$ with $\ind\gamma_i\leq 2$ for $1\leq i\leq 3$, but the bound we obtain by \Cref{C:unitary-biquat} is better in this case.
    Furthermore, if $\exp\alpha=\ind\alpha=4$, then \Cref{T:Wu} does not apply, while \Cref{C:unitary-biquat} does. 
\end{rem}

\section{Bounds on the orthogonal $u$-invariant}\label{section: bounds on the orthogonal u-invariant}
In this section, we assume that $F$ is nonreal with $\cchar F\neq2$.
We study the behaviour of the orthogonal $u$-invariant under multiquadratic field extensions.

\begin{thm}\label{ort-hermit-ind8-wu+our}
   Let $\alpha\in\Br_2(F)$ with $\ind\alpha=8$. Then
\begin{equation*}
    \mbox{$u^+(\alpha)\leq\lfloor\frac{87}{64}u(F)\rfloor+\lfloor\frac{63}{32}u(F)\rfloor\leq \frac{213}{64}u(F)$}.
\end{equation*}	
\end{thm}
\begin{proof}
    By \cite{Row78}, there exists a separable quadratic field extension $K/F$ such that $\ind\alpha_K=4$.
    By \Cref{T:main}, we have  $u^+(\alpha)\leq\sfrac{1}{2} u^+(\alpha_K)+u(\alpha_K/F)$, and 
    by  \Cref{u-inv-behavior-2-ext}, we have $u(K)\leq\frac{3}{2} u(F)$.
    We obtain by \Cref{T:Wu} that $u^+(\alpha_K)\leq\frac{29}{16}u(K)\leq \frac{87}{16}u(F)$.
    Furthermore, by \Cref{C:unitary-biquat}, we have that $u(\alpha_K/F)\leq\frac{63}{32}u(F)$.
    This yields the desired inequality.
\end{proof}

\begin{rem}
    Let $\alpha\in\Br_2(F)$ with $\ind\alpha=8$, as in \Cref{ort-hermit-ind8-wu+our}. 
    By \cite{Tig77}, there exist $\gamma_1,\ldots,\gamma_4\in\Br(F)$ with $\ind\gamma_i=2$ for $1\leq i\leq 4$ and such that $\alpha=\gamma_1+\ldots+\gamma_4$.
    Using this together with \Cref{T:Wu} one obtains that $u^+(\alpha)\leq \frac{1517}{256}u(F)$. 
    This general bound, however, is now considerably improved by \Cref{ort-hermit-ind8-wu+our}.
 
    If, however, we have that $\alpha=\gamma_1+\gamma_2+\gamma_3$ for certain $\gamma_1,\gamma_2,\gamma_3\in\Br(F)$ with $\ind\gamma_1=2$ for $1\leq i\leq 3$, then \Cref{T:Wu} yields that $u^+(\alpha)\leq \frac{197}{64}u(F)$, which is slightly better than the bound obtained in \Cref{ort-hermit-ind8-wu+our}. 
\end{rem}

\begin{prop}\label{orthogonal-hermit-multi}
    Let $\alpha\in\Br(F)$, $n\in\nat^+$ and let $M/F$ be a multiquadratic field extension with $[M:F]=2^n$. 
    There exists a subextension $L/F$ of $M/F$ with $[M:L]=2$ such that
\begin{equation*}
    \ind\alpha \cdot u^+(\alpha)\leq\ind\alpha_M\cdot\big(u^+(\alpha_M)+(3^n-1)\cdot u(\alpha_M/L)\big).
\end{equation*}	
\end{prop}
\begin{proof} 
    We prove the statement by induction on $n$. If $n=1$, then $[M:F]=2$, and we conclude by \Cref{T:main} that the claimed inequality holds with $L=F$.
    Assume now that $n>1$. We fix a quadratic subextension $K/F$ and a multiquadratic subextension $M'/F$ of $M/F$ linearly disjoint from $K/F$ such that $M=M'K$.
    Then $\alpha_K\in\Br(K)$ and $M/K$ is a multiquadratic field extension with $[M:K]=2^{n-1}$. 
    Hence, by the induction hypothesis, there exists a $2$-extension $L_1/K$ contained in $M/K$ with $[M:L_1]=2$ and such that 
\begin{equation*}
    \ind \alpha_K \cdot u^+(\alpha_K)\leq\ind\alpha_M\cdot\big(u^+(\alpha_M)+(3^{n-1}-1)\cdot u(\alpha_M/L_1)\big).
\end{equation*}	  
    Since $M'/F$ is a $2$-extension linearly disjoint from $K/F$ with $[M':F]=2^{n-1}$, it follows by \Cref{2-ext-unitary-hermit-compare} that there exists a subextension $L_2/F$ of $M/F$ linearly disjoint from $K/F$ such that  $L_2K=M'K=M$ and
\begin{equation*}
    \ind\alpha_K\cdot u(\alpha_K/F)\leq 3^{n-1}\cdot \ind\alpha_M\cdot u(\alpha_M/L_2).
\end{equation*}   
    By \Cref{T:main} we have
\begin{equation*}
    \ind\alpha\cdot u^+(\alpha) \leq \ind\alpha_K\cdot\big(u^+(\alpha_K)+2\cdot u(\alpha_K/F)\big)\,.
\end{equation*}
    If $u(\alpha_M/L_1)\geq u(\alpha_M/L_2)$ then we set $L=L_1$, and otherwise we set $L=L_2$. 
    Then $u(\alpha_M/L_i)\leq u(\alpha_M/L)$ for $i=1,2$, and we conclude that 
\begin{equation*} 
    \ind\alpha\cdot u^+(\alpha_K)\leq\ind\alpha_M\cdot\big(u^+(\alpha_M)+(3^{n}-1)\cdot u(\alpha_M/L)\big).\vspace{-5mm}
\end{equation*}
\end{proof}

\begin{cor}\label{orthogonal-hermit-split-multi}
    Let $\alpha\in\Br(F)$ and $n\in\nat$. Let $M/F$ be a multiquadratic field extension such that $[M:F]=2^n$ and $\alpha_M=0$.
    If $u\in\nat$ is such that $u(L)\leq u$ for every subextension $L/F$ of $M/F$, then
\begin{equation*}
    \ind\alpha\cdot u^+(\alpha)\leq\sfrac{3^n+1}2\cdot u.
\end{equation*}	
    In any case, we have 
\begin{equation*}
    \ind \alpha\cdot u^+(\alpha)\leq\sfrac{(3^n+2)3^{n-1}}{2^n}\cdot u(F)\,.
\end{equation*}
\end{cor}
\begin{proof}
    If $n=0$, then $\alpha=0$, whereby $u^+(\alpha)=u^+(F)=u(F)$, so that both parts of the statement hold trivially.
    Assume now that $n\geq 1$.
    By \Cref{orthogonal-hermit-multi}, there exists a $2$-extension $L/F$ contained in $M/F$ with $[M:L]=2$ and such that
\begin{equation*}
    \ind\alpha\cdot u^+(\alpha)\leq\ind\alpha_M\big(u^+(\alpha_M)+(3^n-1)\cdot u(\alpha_M/L)\big).
\end{equation*}
    Since $\alpha_M=0$, we have $\ind\alpha_M=1$, and hence  $u^+(\alpha_M)=u(M)$ and further $u(\alpha_M/L)=u(M/L)\leq \frac{1}{2} u(L)$, if view of \Cref{u-inv-unitary-split-bounds}. 
    This yields the first part.  
    As $M/F$ and $L/F$ are $2$-extensions with $[M:F]=2^n$ and $[L:F]=2^{n-1}$, we obtain by \Cref{u-inv-behavior-2-ext} that 
    $u(M)\leq(\frac{3}{2})^n\cdot u(F)$ and $u(L)\leq(\frac{3}{2})^{n-1}\cdot u(F)$. 
    This yields the second part.
\end{proof}

Most bounds that we presented in this article have strictly weaker hypotheses than previously known bounds. The trade-off is that the bounds that we obtain are comparatively also a bit weaker.

\begin{rem}
    Let $n\in\nat$. Consider the condition on $F$ that, for any $r\in\nat$, every system of $r$ quadratic forms over $F$ in more than $r\cdot 2^n$ variables has a non-trivial zero over $F$.
    With this condition, the proof of \cite[Prop.~3.6]{Mah05} yields that $u^+(\alpha)\leq (1+\frac{1}{\ind \alpha}) \cdot 2^{n-1}$ for any $\alpha\in\Br_2(F)$.

    Note that the condition on systems of quadratic forms also implies that $u(F')\leq 2^n$ for every finite field extension $F'/F$.
    However, the  bound which we get from \cite[Prop.~3.6]{Mah05} is far better than what one would obtain by applying \Cref{orthogonal-hermit-split-multi} with $u=2^n$.

    However, there are fields $F$ for which it is known that $u(F')\leq 2^n$ holds for every finite field extension $F'/F$, while there is no evidence that the stronger condition on systems of quadratic forms over $F$ is satisfied.

    A very interesting such case is that, for $n\geq 3$, of a rational function field 
    $$F=\qq_p(t_1,\dots,t_{n-2})$$ in $n-2$ variables over the field of $p$-adic numbers $\qq_p$ for any prime number $p$.
    Here, it is shown in \cite[Prop.~2.4, Cor.~2.7]{Lee13} that, for any $r\in\nat$, any systems of $r$ quadratic forms over $F$ in more than $r\cdot 2^n$ variables has a solution in some finite extension of odd degree of $F$, and this is further used in \cite[Theorem~3.4]{Lee13} to show that $u(F')\leq 2^n$ for every finite extension $F'/F$.

    Since it is not known whether $u^+(\alpha_{L})=u^+(\alpha)$ for any $\alpha\in\Br_2(F)$ and a finite extension of odd degree $L/F$, the bound from \Cref{orthogonal-hermit-split-multi} is still the best we might have so far.
    For $n=4$, that is, $F=\qq_p(t_1,t_2)$, we obtain for example that $\ind \alpha\cdot u^+(\alpha)\leq (3^7+1)\cdot 8=17504$ for any $\alpha\in\Br_2(F)$.
    From \Cref{T:Wu}, one can get that $u^+(\alpha)\leq 946$, which is better when $\ind \alpha\leq 16$.
    It unknown whether there exists $\alpha\in\Br_2(F)$ with $\ind\alpha>16$ over this field $F$.
\end{rem}

\subsection*{Acknowledgments}
We would like to thank Nicolas Garrel, Archita Mondal and Anne Qu\'equiner-Mathieu for inspiring discussions and various helpful comments and suggestions. 

This work was supported by the Fonds Wetenschappelijk Onderzoek – Vlaanderen in the FWO Odysseus Programme (project G0E6114N, \emph{Explicit Methods in Quadratic Form Theory}), by the FWO-Tournesol programme (project VS05018N), by the Fondazione Cariverona in the programme Ricerca Scientifica di Eccellenza 2018 (project \emph{Reducing complexity in algebra, logic, combinatorics -- REDCOM}), by the 2020 PRIN (project \emph{Derived and underived algebraic stacks and applications}) from MIUR, and by research funds from Scuola Normale Superiore.

\end{document}